\documentclass[preprint,12pt]{elsarticle}
\usepackage[T1]{fontenc}
\usepackage[utf8]{inputenc}
\usepackage{lmodern}
\usepackage{amsmath,amssymb,amsthm,mathtools}
\usepackage{booktabs}
\usepackage{graphicx}
\usepackage{placeins}
\usepackage{tabularx}
\usepackage{enumitem}
\usepackage{microtype}
\usepackage{xcolor} 
\usepackage[colorlinks=true,linkcolor=blue!45!black,
            citecolor=blue!45!black,urlcolor=blue!55!black]{hyperref}
\hypersetup{
  pdftitle={Schatten norms and determinants of linear combinations of matrix tensor powers},
  pdfauthor={Martin Áron Juhász and Mihály Weiner},
  pdfkeywords={tensor powers, Schatten norms, determinants, Schur-Weyl duality, virtual representations}
}

\newtheorem{theorem}{Theorem}[section]
\newtheorem{proposition}[theorem]{Proposition}
\newtheorem{corollary}[theorem]{Corollary}

\theoremstyle{definition}

\theoremstyle{remark}
\newtheorem{remark}[theorem]{Remark}

\DeclareMathOperator{\Sym}{Sym}
\DeclareMathOperator{\Tr}{Tr}
\DeclareMathOperator{\GL}{GL}
\DeclareMathOperator{\U}{U}

\DeclareMathOperator{\sgn}{sgn}
\DeclareMathOperator{\vecop}{vec}
\newcommand{\C}{\mathbb{C}}
\newcommand{\N}{\mathbb{N}}
\newcommand{\Sd}{\mathfrak{S}}
\newcommand{\Rep}{\operatorname{Rep}}
\newcommand{\cS}{\mathbb{S}}
\newcommand{\abs}[1]{\lvert #1\rvert}
\newcommand{\norm}[1]{\left\lVert #1\right\rVert}

\begin{document}
\begin{frontmatter}
\title{Schatten norms and determinants of linear combinations of matrix tensor powers via virtual representations \tnoteref{title}}

\author[bme]{Martin Áron Juhász}
\author[bme]{Mihály Weiner\corref{cor1}}
\cortext[cor1]{Corresponding author}
\address[bme]{Department of Mathematical Analysis, Institute of Mathematics,
Budapest University of Technology and Economics, Műegyetem rkp. 3.,
H-1111 Budapest, Hungary}
\tnotetext[title]{This work was supported by the National Research, Development and Innovation Office of Hungary (NKFIH) via the grants ADVANCED\_25 152599, K 146380 and EXCELLENCE
	151342, and by the Ministry of Culture and Innovation and the NKFIH within the Quantum Information National Laboratory of Hungary (Grant
	No. 2022-2.1.1-NL-2022-00004).}
\begin{abstract}
	Let
	$
	X_n=\sum_{i=1}^s t_i A_i^{\otimes n},
	$	
	where \(A_1,\ldots,A_s\in M_d(\mathbb C)\) and \(t_1,\ldots,t_s\in\mathbb C\) are fixed, while \(n\) grows. Direct computation of determinants or Schatten norms of \(X_n\) is exponential in \(n\). For a single tensor power these quantities are elementary, and even the determinant of a generic two-term combination admits a reduction to polynomially many scalar factors; however, no analogous elementary reduction is available for three or more terms.
	
	We give an exact representation-theoretic method which, for fixed \(d\) and \(s\), computes \(\|X_n\|_p\), \(0<p<\infty\), and determinants in polynomial time in \(n\). Schur--Weyl duality yields a simultaneous block decomposition, while Jacobi--Trudi identities in the Grothendieck ring replace Schur modules by signed combinations of tensor products of symmetric powers.
	
	For \(d=3\), each irreducible contribution reduces to the difference of two explicitly computable symmetric-power terms, leading to an open-source implementation. In a single-thread CPU benchmark, a genuine three-term \(3\times3\) trace-norm problem with \(n=18\) is evaluated in about \(47\) seconds, whereas just storing the unreduced matrix would require approximately \(2.4\times10^{18}\) bytes. Direct and reduced computations agree to relative error below \(3.4\times10^{-15}\) throughout their common range \(n\leq9\).
\end{abstract}

\begin{keyword}
Tensor powers, quantum hypothesis testing, Schatten norms, Schur--Weyl duality,
representation rings, Jacobi--Trudi identity
\MSC[2020] 15A69 \sep 20G05 \sep 65F30 \sep 81P45
\end{keyword}

\end{frontmatter}

\section{Introduction}
\label{sec:introduction}

For fixed matrices $A_1,\ldots,A_s \in M_d(\C)$ and scalars
$t_1,\ldots,t_s\in\C$, consider
\begin{equation}\label{eq:mainquantity}
  X_n=\sum_{i=1}^s t_iA_i^{\otimes n}.
\end{equation}
The order of $X_n$ is $d^n$, and a dense representation contains $d^{2n}$
entries.  Nevertheless, the cases $s=1$, $s=2$, and $s\geq3$ have rather
different characters.

For one tensor-power term, many standard quantities are immediate.  For
example,
\begin{equation}\label{eq:one-term-trivial}
  \norm{A^{\otimes n}}_p=\norm{A}_p^n,
  \qquad
  \det(A^{\otimes n})=\det(A)^{n d^{n-1}}.
\end{equation}
The determinant of a two-term combination is still elementary in a useful
sense.  Suppose that $A_1$ is invertible, put $B=A_1^{-1}A_2$, and let
$\beta_1,\ldots,\beta_d$ be the eigenvalues of $B$, repeated with algebraic
multiplicity.  Factoring out $A_1^{\otimes n}$ and grouping the eigenvalues of
$B^{\otimes n}$ by occupation numbers gives
\begin{align}
 &\det\!\left(t_1A_1^{\otimes n}+t_2A_2^{\otimes n}\right)
 \notag\\
 &\quad=\det(A_1)^{n d^{n-1}}
 \prod_{\substack{\alpha\in\N^d\\\abs{\alpha}=n}}
 \left(t_1+t_2\prod_{j=1}^d\beta_j^{\alpha_j}\right)^{
 n!/(\alpha_1!\cdots\alpha_d!)}.
\label{eq:two-term-determinant}
\end{align}
Thus, for fixed $d$, only $\binom{n+d-1}{d-1}$ distinct scalar factors are
needed.  There is no analogous factorization for three generic matrices:
a third term cannot in general be expressed through the spectrum of one
relative matrix.  This elementary contrast is one indication that the first
genuinely difficult case begins with three tensor-power terms.

Quantum hypothesis testing gives a second indication.  The minimum Bayesian
error probability for distinguishing two states $\rho$ and $\sigma$, with
prior probabilities $\pi_0$ and $\pi_1$, is
\begin{equation}\label{eq:helstrom}
  P_{\mathrm e}(n)
  =\frac12\left(1-
    \norm{\pi_0\rho^{\otimes n}-\pi_1\sigma^{\otimes n}}_1\right);
\end{equation}
see \cite{Helstrom,Holevo}.  Exact finite-copy evaluation is already
nontrivial, but its asymptotic exponential rate has the single-letter
quantum Chernoff formula \cite{AudenaertEtAl}
\begin{equation}\label{eq:quantum-chernoff}
 \lim_{n\to\infty}-\frac1n\log P_{\mathrm e}(n)
 =-\log\inf_{0\leq q\leq1}\Tr\rho^q\sigma^{1-q}.
\end{equation}

Now let the null hypothesis contain two states, $\rho_{0,a}$ and
$\rho_{0,b}$, with prior weights $\pi_{0,a}$ and $\pi_{0,b}$, while the
alternative consists of a state $\sigma$ with weight $\pi_1$.  The optimal
Bayesian error becomes
\begin{equation}\label{eq:composite-helstrom}
 P_{\mathrm e}(n)=\frac12\left(1-
 \norm{\pi_{0,a}\rho_{0,a}^{\otimes n}
      +\pi_{0,b}\rho_{0,b}^{\otimes n}
      -\pi_1\sigma^{\otimes n}}_1\right).
\end{equation}
This is the first three-term trace-norm problem.  In contrast with the simple
case, no general single-letter formula for its asymptotic exponent is known;
indeed, a natural candidate suggested by the classical theory was disproved in \cite{MosonyiSzilagyiWeiner} by one of the present authors and  two collaborators.  The same three-term threshold therefore
appears both in elementary determinant algebra and in the motivating quantum
information problem.

The repeated tensor power in \eqref{eq:mainquantity} has much more symmetry
than a generic $d^n\times d^n$ matrix.  Schur--Weyl duality makes this
symmetry visible: the map $A\mapsto A^{\otimes n}$ splits into polynomial
$\GL(d,\C)$-representations indexed by partitions of $n$.  This observation
alone is not yet a numerical method, because explicit matrices for general
Schur functors are cumbersome.  Our main point is that the required
irreducible matrices need not be constructed.  The Jacobi--Trudi identity
expresses their classes as virtual combinations of tensor products of
symmetric powers.  Since the $p$th power of the Schatten $p$-norm is additive
under orthogonal direct sums, identities in the representation ring become
exact signed identities of numerical quantities.  For determinants and
characteristic polynomials, direct-sum additivity is replaced by
multiplicativity, and virtual coefficients become integer exponents.

Two technical points are essential.  First, singular values are preserved by
unitary, not arbitrary, changes of basis.  We therefore obtain the block
decomposition from the compact group $\U(d)$ and then extend the same fixed
unitary intertwiner to all of $M_d(\C)$.  Second, a virtual representation is
a formal difference, not a matrix with negative blocks.  It is useful here
only because the functional being evaluated is additive on orthogonal direct
sums.

The present implementation supports $2\times2$ and $3\times3$ matrices.
The reproducible benchmarks in Section~\ref{sec:numerics} use a single CPU
thread and include both two-term and genuinely three-term $3\times3$
trace-norm expressions.  For the three-term family, a fresh-process evaluation
at $n=18$ takes about $47.6$ seconds.  Direct dense storage at this tensor power
would already require approximately $2.4\times10^{18}$ bytes.  At $n=9$,
where brute-force comparison is still feasible on the benchmark machine, the
reduced method agrees with direct evaluation to relative error
$3.3\times10^{-15}$ and its warm evaluation is about $7.6\times10^4$ times
faster.  The implementation data extend to $n=30$ for $3\times3$ inputs; at
that power the unreduced matrix has order
\[
  3^{30}=205\,891\,132\,094\,649,
\]
whereas the largest virtual block has order only $18496$.

\subsection*{Main contributions}

The main contribution is the development of a practical new method for the mentioned kind of computations involving large tensor powers. In particular:
\begin{enumerate}[label=(\roman*),leftmargin=2.2em]
  \item We prove a simultaneous unitary Schur--Weyl block decomposition valid
  for every matrix $A\in M_d(\C)$, including singular and nonnormal matrices
  (Theorem~\ref{thm:universal-block}), and show how representation-ring identities yield exact identities for
  additive Schatten functionals and, separately, for multiplicative
  invariants such as determinants and characteristic polynomials
  (Propositions~\ref{prop:additive-functional} and
  \ref{prop:multiplicative-functional}).

  \item We combine determinant twists with Jacobi--Trudi to give a fixed-$d$ master formula involving only symmetric powers
  (Theorem~\ref{thm:master-formula}).  For $d=3$ each Schur contribution is a difference of two explicitly computable terms.

  \item We give an implementation-level description of the sparse monomial tables, coefficient matrices, cache structure, aggregation of identical virtual blocks, and the online evaluation algorithm.

  \item We establish the dense worst-case complexity
  $O(n^{3(d-1)^2+d-1})$, give exact block-size and memory comparisons, report reproducible CPU benchmarks for two- and three-term trace-norm problems, and formulate numerical stability diagnostics for cancellation in the virtual formula.
\end{enumerate}

The original motivation was the composite-hypothesis problem above.  An
earlier project, carried out with Sloan Nietert, treated the $2\times2$ case
in Wolfram Mathematica \cite{WeinerNietertSoftware}.  That dimension is
exceptional: after a determinant twist every relevant irreducible
representation is a symmetric power.  Dimension three is the first case in
which virtual subtraction is genuinely needed.  The present paper develops
that case explicitly and places it in a construction valid for arbitrary
fixed $d$.
\section{A unitary decomposition valid for arbitrary matrices}

We first isolate the passage from the compact group to arbitrary complex
matrices.  This both fixes the Hilbert-space geometry needed for singular
values and explains why the final formulas apply beyond unitary or
invertible inputs.

Let $\lambda=(\lambda_1,\ldots,\lambda_d)$ be a partition of $n$, padded
with zeros, and write $\ell(\lambda)\leq d$ for the condition that at
most $d$ parts are nonzero.  Let $\cS_\lambda(\C^d)$ denote the Schur
module of highest weight $\lambda$, and let $[\lambda]$ be the Specht
module of the symmetric group $\Sd_n$.  Set
\[
  f^\lambda=\dim[\lambda].
\]
One convenient form of the hook-length formula is
\begin{equation}\label{eq:specht-dimension}
  f^\lambda
  =
  n!\,
  \frac{\displaystyle\prod_{1\leq i<j\leq d}
  (\lambda_i-\lambda_j+j-i)}
  {\displaystyle\prod_{i=1}^d(\lambda_i+d-i)!}.
\end{equation}

\begin{theorem}[Simultaneous unitary block decomposition]
\label{thm:universal-block}
For every $d,n\geq1$, there are Hilbert-space structures on the Schur
modules and a unitary map
\[
  W_n:(\C^d)^{\otimes n}\longrightarrow
  \bigoplus_{\substack{\lambda\vdash n\\\ell(\lambda)\leq d}}
  \cS_\lambda(\C^d)\otimes\C^{f^\lambda}
\]
such that, for every $A\in M_d(\C)$,
\begin{equation}\label{eq:universal-block}
  W_n A^{\otimes n}W_n^*
  =
  \bigoplus_{\substack{\lambda\vdash n\\\ell(\lambda)\leq d}}
  \cS_\lambda(A)\otimes I_{f^\lambda}.
\end{equation}
The map $W_n$ depends on $d$ and $n$, but not on $A$.
\end{theorem}

\begin{proof}
Restrict first to $A=U\in\U(d)$.  The actions of $\U(d)$ and $\Sd_n$ on
$(\C^d)^{\otimes n}$ commute.  Complete reducibility for the compact
group, together with Schur--Weyl duality, gives the orthogonal
decomposition
\[
  (\C^d)^{\otimes n}
  \cong
  \bigoplus_{\substack{\lambda\vdash n\\\ell(\lambda)\leq d}}
  \cS_\lambda(\C^d)\otimes[\lambda].
\]
After choosing orthonormal bases, the corresponding intertwiner $W_n$
is unitary, and \eqref{eq:universal-block} holds for every $U\in\U(d)$.

Differentiate the intertwining relation along one-parameter subgroups of
$\U(d)$.  It follows that $W_n$ intertwines the derived representations
of the Lie algebra $\mathfrak{u}(d)$.  Since the complex linear span of
$\mathfrak{u}(d)$ is $\mathfrak{gl}_d(\C)$, it intertwines the
complexified derived representations as well.  The group
$\GL(d,\C)$ is connected, so the relation follows for every invertible
complex matrix.  Equivalently, one may use that every invertible complex
matrix has a logarithm and exponentiate the derived relation.

Finally, both sides of \eqref{eq:universal-block} are matrices whose
entries are polynomials in the entries of $A$.  Since $\GL(d,\C)$ is
dense in $M_d(\C)$, the identity extends to singular matrices.
\end{proof}

\begin{remark}[Why $\U(d)$ rather than $\GL(d,\C)$ or $\mathrm{SU}(d)$?]
\label{rem:why-unitary}
Starting from $\GL(d,\C)$ gives the correct algebraic decomposition but
does not by itself supply an orthogonal one: a general intertwining
matrix does not preserve singular values.  Compactness of $\U(d)$ gives
invariant Hermitian inner products and hence a unitary intertwiner.  The
group $\mathrm{SU}(d)$ is compact as well, but $\U(d)$ is more natural
here because determinant characters remain visible.  They are exactly
what separates polynomial $\GL(d,\C)$-representations that become
equivalent after restriction to $\mathrm{SU}(d)$.  Theorem
\ref{thm:universal-block} then transfers the unitary decomposition from
$\U(d)$ to all complex matrices.
\end{remark}

For $0<p<\infty$, the Schatten $p$-functional is
\[
   \norm{T}_p^p=\Tr\!\left((T^*T)^{p/2}\right).
\]
For $0<p<1$ this is the $p$th power of a quasi-norm; its direct-sum
additivity is unchanged.

\begin{corollary}\label{cor:schurweyl-norm}
Let $A_1,\ldots,A_s\in M_d(\C)$, $t_1,\ldots,t_s\in\C$, and
$0<p<\infty$.  Then
\begin{equation}\label{eq:schurweyl-norm}
 \norm{\sum_{i=1}^s t_iA_i^{\otimes n}}_p^p
 =
 \sum_{\substack{\lambda\vdash n\\\ell(\lambda)\leq d}}
 f^\lambda
 \norm{\sum_{i=1}^s t_i\cS_\lambda(A_i)}_p^p.
\end{equation}
\end{corollary}

\begin{proof}
Apply \eqref{eq:universal-block} to every $A_i$, take the same linear
combination, and use
\[
  \norm{T\otimes I_m}_p^p=m\norm{T}_p^p
  \quad\text{and}\quad
  \norm{T\oplus R}_p^p=\norm{T}_p^p+\norm{R}_p^p.
\]
\end{proof}

\begin{remark}[The operator norm]
For $p=\infty$, an actual orthogonal decomposition still gives
\[
 \norm{\bigoplus_j T_j}_\infty=\max_j\norm{T_j}_\infty.
\]
The virtual-representation step below, however, relies on additivity and
does not extend to the maximum.  We therefore restrict the virtual
formula and its algorithmic consequences to $0<p<\infty$.
\end{remark}

\section{Additive spectral functionals and virtual representations}

We now explain precisely how a representation-ring identity can be used
without constructing an irreducible representation.

Fix $A_1,\ldots,A_s\in M_d(\C)$, coefficients
$t_1,\ldots,t_s\in\C$, and $0<p<\infty$.  If $\rho$ is a homogeneous
polynomial representation of $\GL(d,\C)$, choose a Hermitian inner
product for which $\rho|_{\U(d)}$ is unitary and define
\begin{equation}\label{eq:phi-definition}
   \Phi_p(\rho)
   =
   \norm{\sum_{i=1}^s t_i\rho(A_i)}_p^p.
\end{equation}
The polynomial map $\rho(A)$ is defined at singular $A$ as well.

\begin{proposition}\label{prop:additive-functional}
The quantity $\Phi_p(\rho)$ depends only on the isomorphism class of
$\rho$ and is additive:
\[
  \Phi_p(\rho\oplus\tau)=\Phi_p(\rho)+\Phi_p(\tau).
\]
Consequently, it extends uniquely to a group homomorphism
\[
  \Phi_p:K_0(\Rep_{\mathrm{poly}}\GL(d,\C))\longrightarrow\mathbb{R},
\]
from the Grothendieck group of finite-dimensional polynomial
representations.
\end{proposition}

\begin{proof}
Two isomorphic polynomial representations restrict to equivalent
unitary representations of $\U(d)$.  Their equivalence can therefore
be implemented by a unitary intertwiner.  By the same complexification
and polynomial-continuation argument as in
Theorem~\ref{thm:universal-block}, that intertwiner works for every
$A\in M_d(\C)$.  Hence the matrices in \eqref{eq:phi-definition} are
unitarily similar and have the same singular values.

For a direct sum,
\[
 \sum_i t_i(\rho\oplus\tau)(A_i)
 =
 \left(\sum_i t_i\rho(A_i)\right)
 \oplus
 \left(\sum_i t_i\tau(A_i)\right),
\]
so additivity follows from direct-sum additivity of the $p$th power of
the Schatten functional.  The universal property of the Grothendieck
group gives the extension.
\end{proof}

\begin{remark}
If $[\rho]=[\rho_+]-[\rho_-]$ in the representation ring, then
\[
  \Phi_p(\rho)=\Phi_p(\rho_+)-\Phi_p(\rho_-).
\]
This statement concerns the value of an additive functional.  It does
not assert that $\rho(A)$ is obtained by deleting a visible matrix block
from $\rho_+(A)$ in the bases used for computation.
\end{remark}

\subsection{Multiplicative invariants}

The representation-ring method is not limited to additive functionals.
For a polynomial representation $\rho$, define
\begin{equation}\label{eq:delta-definition}
  \Delta(\rho)
  =
  \det\!\left(\sum_{i=1}^s t_i\rho(A_i)\right).
\end{equation}
Unlike $\Phi_p$, this functional is multiplicative under direct sums:
\[
  \Delta(\rho\oplus\tau)=\Delta(\rho)\Delta(\tau).
\]
Thus a virtual coefficient becomes an exponent rather than a scalar
coefficient.

\begin{proposition}\label{prop:multiplicative-functional}
Suppose that
\begin{equation}\label{eq:general-virtual-identity}
  [\rho]=\sum_{j=1}^q c_j[\rho_j],
  \qquad c_j\in\mathbb{Z},
\end{equation}
in the Grothendieck group of polynomial representations.  Then
\begin{equation}\label{eq:det-cross-multiplied}
  \Delta(\rho)
  \prod_{c_j<0}\Delta(\rho_j)^{-c_j}
  =
  \prod_{c_j>0}\Delta(\rho_j)^{c_j}.
\end{equation}
In particular, whenever all factors with $c_j<0$ are nonzero,
\begin{equation}\label{eq:det-quotient}
  \Delta(\rho)=\prod_{j=1}^q\Delta(\rho_j)^{c_j}.
\end{equation}
The same statements hold with $\Delta(\rho)$ replaced by the
characteristic polynomial
\[
  \det\!\left(zI-\sum_i t_i\rho(A_i)\right).
\]
\end{proposition}

\begin{proof}
Separate the positive and negative parts of
\eqref{eq:general-virtual-identity}.  The definition of the Grothendieck
group gives an isomorphism
\[
  \rho\oplus
  \bigoplus_{c_j<0}\rho_j^{\oplus(-c_j)}
  \cong
  \bigoplus_{c_j>0}\rho_j^{\oplus c_j},
\]
possibly after adding the same auxiliary representation to both sides.
Apply the determinant to the corresponding direct-sum matrices, treating
the entries of the $A_i$ and the $t_i$ as indeterminates.  The auxiliary
factor is not the zero polynomial---set $t_1=1$, $A_1=I$, and all other
$t_i=0$---so it cancels in the polynomial ring.  This gives
\eqref{eq:det-cross-multiplied}.  Division gives
\eqref{eq:det-quotient} under the stated nonvanishing condition.  The
argument with $zI$ is identical.
\end{proof}

\begin{remark}
The integrality of the coefficients in
\eqref{eq:general-virtual-identity} is exactly what makes negative
coefficients harmless for a multiplicative invariant: they become
negative integer exponents.  At exceptional inputs a direct quotient may
have the indeterminate form $0/0$.  The cross-multiplied identity remains
valid there; numerical evaluation can use a small generic perturbation
and a limit, while symbolic evaluation can cancel common polynomial
factors before substitution.
\end{remark}

\section{Jacobi--Trudi reduction to symmetric powers}

Let $\lambda=(\lambda_1,\ldots,\lambda_d)$ be a partition.  Put
\[
  m=\lambda_d,
  \qquad
  \mu=\lambda-m(1,\ldots,1).
\]
Then $\mu_d=0$, so $r=\ell(\mu)\leq d-1$, and
\begin{equation}\label{eq:det-twist}
   \cS_\lambda(A)=\det(A)^m\cS_\mu(A)
   \qquad (A\in M_d(\C)).
\end{equation}
The identity remains meaningful for singular $A$: both sides are
polynomial functions of its entries.

For $q\geq0$, write $H_q=[\Sym^q(\C^d)]$ in the representation ring,
put $H_0=1$, and put $H_q=0$ for $q<0$.  The Jacobi--Trudi identity
\cite[Ch.~I]{Macdonald} gives
\begin{equation}\label{eq:JT}
   [\cS_\mu(\C^d)]
   =
   \det\!\left(H_{\mu_i-i+j}\right)_{i,j=1}^r.
\end{equation}
Expanding the determinant yields
\begin{equation}\label{eq:JT-expanded}
 [\cS_\mu(\C^d)]
 =
 \sum_{\sigma\in\Sd_r}\sgn(\sigma)
 \left[
   \bigotimes_{j=1}^r
   \Sym^{\,\mu_j-j+\sigma(j)}(\C^d)
 \right],
\end{equation}
where a tensor product containing a negative symmetric power is
interpreted as zero.

For $\sigma\in\Sd_r$, define the polynomial representation
\begin{equation}\label{eq:B-def}
  B_{\lambda,\sigma}(A)
  =
  \det(A)^m
  \bigotimes_{j=1}^r
  \Sym^{\,\mu_j-j+\sigma(j)}(A),
\end{equation}
again setting it equal to zero if one of the exponents is negative.
When $\mu=0$, the empty tensor product is the one-dimensional trivial
representation.

\begin{theorem}[Master formula]\label{thm:master-formula}
Let $A_1,\ldots,A_s\in M_d(\C)$, let
$t_1,\ldots,t_s\in\C$, let $n\geq1$, and let $0<p<\infty$.  Then
\begin{align}
 &\norm{\sum_{i=1}^s t_i A_i^{\otimes n}}_p^p \notag\\
 &\quad =
 \sum_{\substack{\lambda\vdash n\\\ell(\lambda)\leq d}}
 f^\lambda
 \sum_{\sigma\in\Sd_{\ell(\mu)}}
 \sgn(\sigma)
 \norm{\sum_{i=1}^s t_i B_{\lambda,\sigma}(A_i)}_p^p,
 \label{eq:master-formula}
\end{align}
where $\mu=\lambda-\lambda_d(1,\ldots,1)$ and
$B_{\lambda,\sigma}$ is defined by \eqref{eq:B-def}.
\end{theorem}

\begin{proof}
For a fixed partition $\lambda$, combine the determinant twist
\eqref{eq:det-twist}, the Jacobi--Trudi identity
\eqref{eq:JT-expanded}, and Proposition
\ref{prop:additive-functional}.  This gives
\[
 \norm{\sum_i t_i\cS_\lambda(A_i)}_p^p
 =
 \sum_{\sigma\in\Sd_{\ell(\mu)}}\sgn(\sigma)
 \norm{\sum_i t_i B_{\lambda,\sigma}(A_i)}_p^p.
\]
Substitution in \eqref{eq:schurweyl-norm} proves the result.
\end{proof}

\begin{remark}
Formula \eqref{eq:master-formula} is an exact algebraic reduction.
When its blocks are evaluated in floating-point arithmetic, the final
answer is of course a numerical approximation.  In particular, signed
cancellation in the inner sum must be monitored; see
Section~\ref{sec:numerics}.
\end{remark}

\begin{corollary}[Determinants from the master decomposition]
\label{cor:det-master}
Put
\[
  \Delta_\lambda
  =
  \det\!\left(\sum_i t_i\cS_\lambda(A_i)\right),
  \qquad
  \Delta_{\lambda,\sigma}
  =
  \det\!\left(\sum_i t_iB_{\lambda,\sigma}(A_i)\right).
\]
Then
\begin{equation}\label{eq:full-det-schur}
 \det\!\left(\sum_i t_iA_i^{\otimes n}\right)
 =
 \prod_{\substack{\lambda\vdash n\\\ell(\lambda)\leq d}}
 \Delta_\lambda^{f^\lambda},
\end{equation}
and each Schur factor satisfies
\begin{equation}\label{eq:schur-det-cross}
  \Delta_\lambda
  \prod_{\substack{\sigma\in\Sd_{\ell(\mu)}\\
                    \sgn(\sigma)=-1}}
  \Delta_{\lambda,\sigma}
  =
  \prod_{\substack{\sigma\in\Sd_{\ell(\mu)}\\
                    \sgn(\sigma)=1}}
  \Delta_{\lambda,\sigma}.
\end{equation}
For generic input matrices, the factors on the left are nonzero and
\eqref{eq:schur-det-cross} gives $\Delta_\lambda$ as a quotient.
\end{corollary}

\begin{proof}
Equation \eqref{eq:full-det-schur} follows from the actual orthogonal
direct sum in Theorem~\ref{thm:universal-block}.  Equation
\eqref{eq:schur-det-cross} follows from the Jacobi--Trudi identity and
Proposition~\ref{prop:multiplicative-functional}.
\end{proof}

\section{The explicit formulas for \texorpdfstring{\(d=2\) and \(d=3\)}
{d=2 and d=3}}

\subsection{Why the \texorpdfstring{\(2\times2\)}{2 by 2} case is exceptional}

For $d=2$, every $\lambda\vdash n$ has the form
$\lambda=(a+m,m)$ with $a,m\geq0$, and
\begin{equation}\label{eq:d2}
  \cS_{(a+m,m)}(A)
  =
  \det(A)^m\Sym^a(A).
\end{equation}
There is no nontrivial Jacobi--Trudi subtraction.  Thus symmetric powers
and determinant characters already give every block directly.  This
special feature motivated the earlier Mathematica implementation
\cite{WeinerNietertSoftware}, but it does not persist in dimension
three.

\subsection{The first virtual case:
\texorpdfstring{\(3\times3\)}{3 by 3} matrices}

Let $\lambda=(\lambda_1,\lambda_2,\lambda_3)\vdash n$ and set
\[
   m=\lambda_3,\qquad
   a=\lambda_1-\lambda_3,\qquad
   b=\lambda_2-\lambda_3.
\]
The $2\times2$ Jacobi--Trudi determinant gives
\begin{equation}\label{eq:d3-virtual}
 [\cS_\lambda]
 =
 [\det^m\otimes\Sym^a\otimes\Sym^b]
 -
 [\det^m\otimes\Sym^{a+1}\otimes\Sym^{b-1}],
\end{equation}
where the second term is zero when $b=0$.  Therefore
\begin{align}
 &\norm{\sum_{i=1}^s t_i\cS_\lambda(A_i)}_p^p
 \notag\\
 &=
 \norm{
   \sum_{i=1}^s
   t_i\det(A_i)^m
   \bigl(\Sym^a(A_i)\otimes\Sym^b(A_i)\bigr)
 }_p^p
 \notag\\
 &\quad -
 \norm{
   \sum_{i=1}^s
   t_i\det(A_i)^m
   \bigl(\Sym^{a+1}(A_i)\otimes\Sym^{b-1}(A_i)\bigr)
 }_p^p.
\label{eq:d3-functional}
\end{align}
Together with \eqref{eq:schurweyl-norm}, this is the formula used by the
$3\times3$ implementation.

The appearance of a difference in \eqref{eq:d3-functional} is not a
numerical heuristic.  By the Pieri rule, the first representation in
\eqref{eq:d3-virtual} contains $\cS_\lambda$ and the second
representation as orthogonal summands after restriction to $\U(3)$.
The difference of the additive spectral functionals is therefore exact.

\section{Explicit matrices for symmetric powers}
\label{sec:symmetric-matrices}

It remains to construct $\Sym^k(A)$ efficiently.  Let
\[
  \mathcal{I}_{d,k}
  =
  \left\{\alpha=(\alpha_1,\ldots,\alpha_d)\in\N^d:
  \abs{\alpha}:=\sum_{j=1}^d\alpha_j=k\right\}.
\]
The normalized occupation-number vectors indexed by
$\mathcal{I}_{d,k}$ form an orthonormal basis of
$\Sym^k(\C^d)$.  Its dimension is
\begin{equation}\label{eq:sym-dimension}
  D_{d,k}
  =
  \binom{k+d-1}{d-1}.
\end{equation}

For $\alpha,\beta\in\mathcal{I}_{d,k}$, let
$\mathcal{E}(\alpha,\beta)$ be the set of nonnegative integer matrices
$E=(E_{uv})_{u,v=1}^d$ whose row sums are $\alpha$ and whose column sums
are $\beta$.  Write
\[
  \alpha!=\prod_{u=1}^d\alpha_u!,
  \qquad
  \beta!=\prod_{v=1}^d\beta_v!,
  \qquad
  E!=\prod_{u,v=1}^d E_{uv}!.
\]

\begin{proposition}\label{prop:sym-entry}
In the normalized occupation-number basis,
\begin{equation}\label{eq:sym-entry}
 \bigl[\Sym^k(A)\bigr]_{\alpha,\beta}
 =
 \sqrt{\alpha!\,\beta!}
 \sum_{E\in\mathcal{E}(\alpha,\beta)}
 \frac{\displaystyle\prod_{u,v=1}^d A_{uv}^{E_{uv}}}{E!}.
\end{equation}
\end{proposition}

\begin{proof}
Expand $A^{\otimes k}$ between the normalized symmetrizations of tensor
words with occupation vectors $\beta$ and $\alpha$.  A pair of words
contributes the monomial $\prod_{u,v}A_{uv}^{E_{uv}}$, where $E_{uv}$
counts positions at which the input letter is $v$ and the output letter
is $u$.  For fixed $E$, the number of such assignments is
$k!/E!$.  The two normalization factors cancel the $k!$ and leave
\eqref{eq:sym-entry}.
\end{proof}

For fixed $d$ and $k$, the sets of exponent arrays $E$ and the
coefficients $\sqrt{\alpha!\beta!}/E!$ do not depend on $A$.  They can
therefore be precomputed and stored sparsely.  Evaluating
$\Sym^k(A)$ then reduces to evaluating a fixed list of monomials in the
$d^2$ entries of $A$ and accumulating them into the prescribed matrix
positions.  The same precomputed table is reused for every matrix in a
numerical experiment.

\subsection{Sparse coefficient tables and online evaluation}
\label{sec:sparse-implementation}

Formula \eqref{eq:sym-entry} is suitable for direct implementation, but it is
wasteful to enumerate the contingency tables $E$ every time a new numerical
matrix is supplied.  The matrix-independent part can be extracted completely.
Let
\[
 \Gamma_{d,k}=\left\{\gamma\in\N^{d^2}:\abs{\gamma}=k\right\},
 \qquad
 M_{d,k}=\abs{\Gamma_{d,k}}=\binom{k+d^2-1}{d^2-1},
\]
and fix an ordering of $\Gamma_{d,k}$.  For $A=(A_{uv})$, define the monomial
vector $m_{d,k}(A)\in\C^{M_{d,k}}$ by
\[
  [m_{d,k}(A)]_\gamma=\prod_{u,v=1}^d A_{uv}^{\gamma_{uv}}.
\]
After flattening the symmetric-power matrix in a fixed row-major order,
Proposition~\ref{prop:sym-entry} gives a sparse matrix $C_{d,k}$, independent
of $A$, such that
\begin{equation}\label{eq:sparse-evaluation}
  \vecop\!\left(\Sym^k(A)\right)=C_{d,k}m_{d,k}(A).
\end{equation}
The rows of $C_{d,k}$ are indexed by pairs $(\alpha,\beta)$ of occupation
vectors, and a nonzero coefficient in column $\gamma$ is obtained by summing
$\sqrt{\alpha!\beta!}/E!$ over the exponent matrices $E$ with flattened
exponent vector $\gamma$, row sums $\alpha$, and column sums $\beta$.

The implementation stores the ordered exponent array and $C_{d,k}$ in sparse
compressed form.  At run time, all monomials are evaluated by vectorized
power and product operations, followed by one sparse matrix--vector product
and a reshape.  Thus symbolic expressions are absent from the online phase.
For a fixed query $A_1,\ldots,A_s$, each matrix $\Sym^k(A_i)$ is cached the
first time degree $k$ is requested and is then reused across all partitions
and Jacobi--Trudi terms.

For $d=3$, one can also aggregate identical virtual blocks before any dense
matrix is formed.  For
$\lambda=(a+m,b+m,m)$, formula \eqref{eq:d3-functional} contributes the key
$(m,a,b)$ with integer weight $f^\lambda$ and, when $b>0$, the key
$(m,a+1,b-1)$ with weight $-f^\lambda$.  Let $c_{m,u,v}$ be the total weight
after summing over all partitions of $n$.  Then
\begin{align}
 \norm{X_n}_p^p
 =\sum_{(m,u,v)} c_{m,u,v}
 \norm{\sum_{i=1}^s t_i\det(A_i)^m
       \bigl(\Sym^u(A_i)\otimes\Sym^v(A_i)\bigr)}_p^p.
\label{eq:d3-aggregated}
\end{align}
Zero coefficients are discarded.  This aggregation is exact integer
arithmetic and avoids repeated evaluation of coincident positive and negative
terms.

The online $3\times3$ algorithm can therefore be summarized as follows.

\begin{quote}
\small
\textbf{Algorithm 1 (weighted Schatten functional for $d=3$).}
Given $n,p$, matrices $A_i$, and coefficients $t_i$:
\begin{enumerate}[leftmargin=1.8em]
 \item enumerate the partitions $\lambda\vdash n$ with at most three parts,
       compute $f^\lambda$, and assemble the integer dictionary
       $(m,u,v)\mapsto c_{m,u,v}$;
 \item determine all symmetric degrees required by the nonzero keys, evaluate
       and cache $\Sym^k(A_i)$ using \eqref{eq:sparse-evaluation};
 \item for each nonzero key, form the weighted dense block in
       \eqref{eq:d3-aggregated}, compute its singular values, and add
       $c_{m,u,v}\sum_j s_j^p$ to the accumulator;
 \item take the $p$th root after checking the cancellation diagnostics in
       Section~\ref{sec:numerics}.
\end{enumerate}
\end{quote}

The matrix-independent tables are an offline resource.  Their generation,
loading, and caching must be distinguished from repeated online evaluation in
any meaningful benchmark.  The public implementation follows this separation
and exposes a class-level interface for block decompositions and weighted
Schatten norms \cite{TensorpowSoftware}.

\section{Algorithmic complexity}
\label{sec:complexity}

For clarity, the general-$d$ evaluation of \eqref{eq:master-formula} consists
of the following steps.

\begin{enumerate}[leftmargin=2.2em]
  \item Enumerate the partitions $\lambda\vdash n$ with
  $\ell(\lambda)\leq d$ and compute $f^\lambda$ from
  \eqref{eq:specht-dimension}.

  \item Remove the last part $m=\lambda_d$ and form
  $\mu=\lambda-m(1,\ldots,1)$.

  \item For each nonzero Jacobi--Trudi term
  $\sigma\in\Sd_{\ell(\mu)}$, obtain the required symmetric powers from the
  cache, take their Kronecker products, and multiply by $\det(A_i)^m$.

  \item Form $\sum_i t_iB_{\lambda,\sigma}(A_i)$, compute its singular
  values, and add the signed $p$th powers with weight $f^\lambda$.

  \item Aggregate the contributions and take the $p$th root.
\end{enumerate}

We give a deliberately conservative dense-arithmetic bound for the online
phase.  The number of partitions of $n$ with at most $d$ parts is
$O(n^{d-1})$.  After the determinant twist, each Jacobi--Trudi term has at
most $d-1$ symmetric factors, and there are at most $(d-1)!$ such terms.  For
fixed $d$ and degrees $O(n)$,
\[
  \dim\Sym^k(\C^d)=\binom{k+d-1}{d-1}=O(n^{d-1}).
\]
Consequently, the largest dense matrix in a Jacobi--Trudi term has order
\begin{equation}\label{eq:max-block-size}
  N_{\max}=O\!\left(n^{(d-1)^2}\right).
\end{equation}

\begin{theorem}\label{thm:complexity}
For fixed $d$, fixed $s$, and $0<p<\infty$, a direct dense implementation of
\eqref{eq:master-formula} has worst-case online arithmetic complexity
\begin{equation}\label{eq:general-complexity}
  O\!\left(n^{3(d-1)^2+d-1}\right)
\end{equation}
and peak dense-matrix storage
\begin{equation}\label{eq:general-memory}
  O\!\left(n^{2(d-1)^2}\right),
\end{equation}
apart from the matrix-independent sparse coefficient tables.
\end{theorem}

\begin{proof}
There are $O(n^{d-1})$ partitions and, for fixed $d$, only a constant number
of Jacobi--Trudi terms per partition.  Dense singular-value decomposition of
a matrix of order \eqref{eq:max-block-size} costs
$O(n^{3(d-1)^2})$.  Processing the terms sequentially gives the asserted
storage bound.  Caching the symmetric powers adds lower-order dense storage
for fixed $d$ and does not change the exponent.
\end{proof}

Dense determinant evaluation by elimination has the same cubic arithmetic
order.  At generic inputs, where the quotient form of
Corollary~\ref{cor:det-master} is directly evaluable, determinants therefore
obey the same worst-case polynomial bound.  Characteristic polynomials can
also be recovered from the multiplicative identities, although their full
symbolic expansion may of course be substantially more expensive than a
single numerical determinant.

For $d=3$, the largest block has order $O(n^4)$, the dense SVD cost per block
is $O(n^{12})$, and the number of partitions is $O(n^2)$.  The resulting
bound is $O(n^{14})$, with peak dense storage $O(n^8)$.  Most partitions
produce much smaller blocks, so this worst-case exponent is not a prediction
of moderate-$n$ wall-clock behavior.

\begin{remark}
The qualification ``for fixed $d$'' is essential.  Both the number of
Jacobi--Trudi terms and the polynomial degree in
\eqref{eq:general-complexity} grow with $d$.  The result is a
fixed-dimension polynomial algorithm, not a polynomial-time algorithm when
$d$ is part of the input.
\end{remark}

\section{Numerical realization, benchmarks, and validation}
\label{sec:numerics}

The formulas have been implemented in the open-source Python package
\texttt{tensorpow} \cite{TensorpowSoftware}.  The benchmarked release,
version 0.3.0, supports $2\times2$ and $3\times3$ inputs, with representation
data through degree $79$ for the former and degree $30$ for the latter.  The
data can be cached locally before a computation, and the implementation never
forms the full $d^n\times d^n$ matrix.  The package also provides an optional
GPU backend for the dense singular-value decompositions; the benchmarks below
use the CPU backend so that all reported timings refer to one controlled
hardware configuration.

\subsection{Exact size comparison}

Table~\ref{tab:size-comparison} gives a machine-independent comparison for
$d=3$.  The full storage column assumes dense \texttt{complex128} data,
namely 16 bytes per entry.  The maximum virtual-block order is obtained from
\[
 \max_{a+b+3m=n}\left\{
   \binom{a+2}{2}\binom{b+2}{2},
   \binom{a+3}{2}\binom{b+1}{2}
 \right\},
\]
where the second term is included only for $b\geq1$.  These values do not
depend on a particular matrix or software implementation.

\begin{table}[ht]
\centering
\caption{Full tensor space versus the largest $3\times3$ virtual block.}
\label{tab:size-comparison}
\begin{tabular}{rrrrr}
\toprule
$n$ & full order $3^n$ & full dense storage & partitions & max. block order \\
\midrule
9  & $19\,683$             & $6.20\times10^{9}$ B  & 12 & 315 \\
12 & $531\,441$            & $4.52\times10^{12}$ B & 19 & 784 \\
15 & $14\,348\,907$        & $3.29\times10^{15}$ B & 27 & 1620 \\
18 & $387\,420\,489$       & $2.40\times10^{18}$ B & 37 & 3025 \\
21 & $10\,460\,353\,203$   & $1.75\times10^{21}$ B & 48 & 5148 \\
24 & $282\,429\,536\,481$  & $1.28\times10^{24}$ B & 61 & 8281 \\
26 & $2\,541\,865\,828\,329$ & $1.03\times10^{26}$ B & 70 & 11025 \\
30 & $205\,891\,132\,094\,649$ & $6.78\times10^{29}$ B & 91 & 18496 \\
\bottomrule
\end{tabular}
\end{table}

Already at $n=12$, direct dense storage would require several terabytes,
whereas the largest reduced block has order below $10^3$. At $n=30$, the
reduction replaces a matrix with about $4.24\times10^{28}$ entries by 91
signed block contributions, none larger than $18496\times18496$.

\subsection{Reproducible CPU benchmark}
\label{sec:benchmark}

The benchmark suite compares the virtual-block method with direct Kronecker
construction wherever the latter remains feasible.  It evaluates two fixed
families,
\begin{align}
 X_n^{(2)}&=\frac12A^{\otimes n}-\frac12B^{\otimes n},
 \label{eq:benchmark-two-term}\\
 X_n^{(3)}&=\frac14A^{\otimes n}+\frac14B^{\otimes n}
             -\frac12C^{\otimes n},
 \label{eq:benchmark-three-term}
\end{align}
where $A,B,C$ are independently generated random $3\times3$ density matrices
from the fixed seed $20260729$.  Each state is obtained as $GG^*/\Tr(GG^*)$
from a complex Gaussian matrix $G$.  The direct method explicitly constructs
the Kronecker powers and computes a dense SVD.  The virtual-block method is
measured both in a warm mode, where a \texttt{TensorPowerCalculator} instance
and cached data are reused, and in a cold-process mode, where each measurement
starts a fresh Python process.

The measurements were made on CachyOS Linux on an Intel Core i5-14600KF at
base clock speed with 32 GB DDR5-6000 CL36 memory.  All BLAS/LAPACK thread
environment variables were fixed to one.  The software stack was Python
3.13.14, \texttt{tensorpow} 0.4.2, NumPy 2.5.2, SciPy 1.18.0, and OpenBLAS
0.3.34.  Representation files were cached before timing, so no network I/O is
included.  Warm virtual-block timings use three repetitions after one warm-up;
Table~\ref{tab:benchmark-selected} reports their medians and median absolute
deviations (MAD).  Direct and cold-process measurements use one repetition at
each $n$.  The complete raw output, environment metadata, tables, and plots are available on GitHub \cite{TensorpowBenchmark}, while the precomputed symmetric representations are archived on Zenodo \cite{TensorpowZenodoData}.

\begin{table}[ht]
\centering
\caption{Selected trace-norm benchmark results in seconds.  Warm virtual-block
entries are median $\pm$ MAD over three repetitions.  The last column is the
relative discrepancy between the direct and three-term virtual-block results;
a dash means that the direct method was not run.}
\label{tab:benchmark-selected}
\scriptsize
\setlength{\tabcolsep}{3.5pt}
\begin{tabular}{r r r r r r}
\toprule
$n$ & Direct, 3-term & VB warm, 2-term & VB warm, 3-term & VB cold, 3-term & Rel. err. \\
\midrule
5  & $6.11\!\times\!10^{-3}$ & $1.10\!\times\!10^{-3}\pm3.6\!\times\!10^{-5}$ & $1.31\!\times\!10^{-3}\pm8.1\!\times\!10^{-5}$ & $0.167$ & $4.1\!\times\!10^{-16}$ \\
7  & $3.66$                  & $4.69\!\times\!10^{-3}\pm1.0\!\times\!10^{-4}$ & $5.71\!\times\!10^{-3}\pm6.4\!\times\!10^{-6}$ & $0.173$ & $1.1\!\times\!10^{-15}$ \\
8  & $106.25$                & $0.0158\pm1.6\!\times\!10^{-5}$ & $0.0185\pm8.6\!\times\!10^{-5}$ & $0.187$ & $1.6\!\times\!10^{-15}$ \\
9  & $2884.90$               & $0.0325\pm1.1\!\times\!10^{-6}$ & $0.0377\pm1.1\!\times\!10^{-4}$ & $0.209$ & $3.3\!\times\!10^{-15}$ \\
12 & ---                     & $0.468\pm3.7\!\times\!10^{-4}$ & $0.520\pm8.0\!\times\!10^{-4}$ & $0.702$ & --- \\
15 & ---                     & $4.63\pm0.011$ & $4.93\pm0.0049$ & $5.20$ & --- \\
18 & ---                     & $45.43\pm0.015$ & $47.15\pm0.027$ & $47.57$ & --- \\
22 & ---                     & $480.78\pm0.016$ & $494.01\pm0.178$ & $488.44$ & --- \\
\bottomrule
\end{tabular}
\end{table}

Figure~\ref{fig:benchmark-plots} shows the full available three-term
series.  The direct method becomes prohibitively expensive between $n=8$ and
$n=9$: its runtime rises from about $106$ seconds to about $2885$ seconds,
whereas the warm virtual-block time at $n=9$ is $0.0377$ seconds.  Agreement
through the entire overlap range $2\leq n\leq9$ is at relative error at most
$3.4\times10^{-15}$.  Beyond the direct range, the three-term virtual-block
calculation remains below one second at $n=12$, below five seconds at $n=15$,
and below one minute at $n=18$.  At $n=18$ the full $3^n\times3^n$ matrix
would contain approximately $1.50\times10^{17}$ entries and occupy about
$2.4\times10^{18}$ bytes in complex double precision.

\begin{figure}[ht]
\centering
\includegraphics[width=0.96\textwidth]{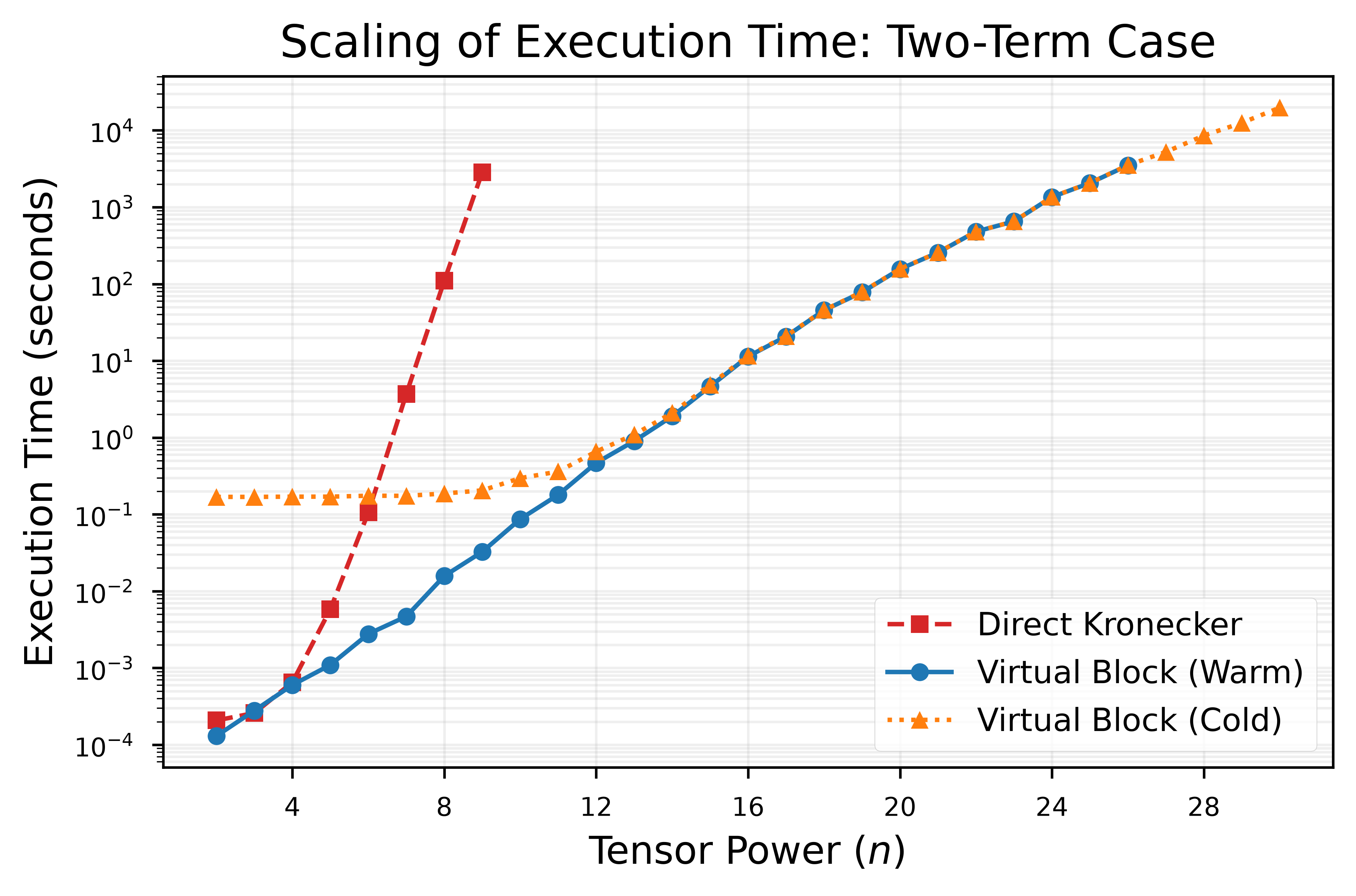}

\vspace{1.5em}

\includegraphics[width=0.96\textwidth]{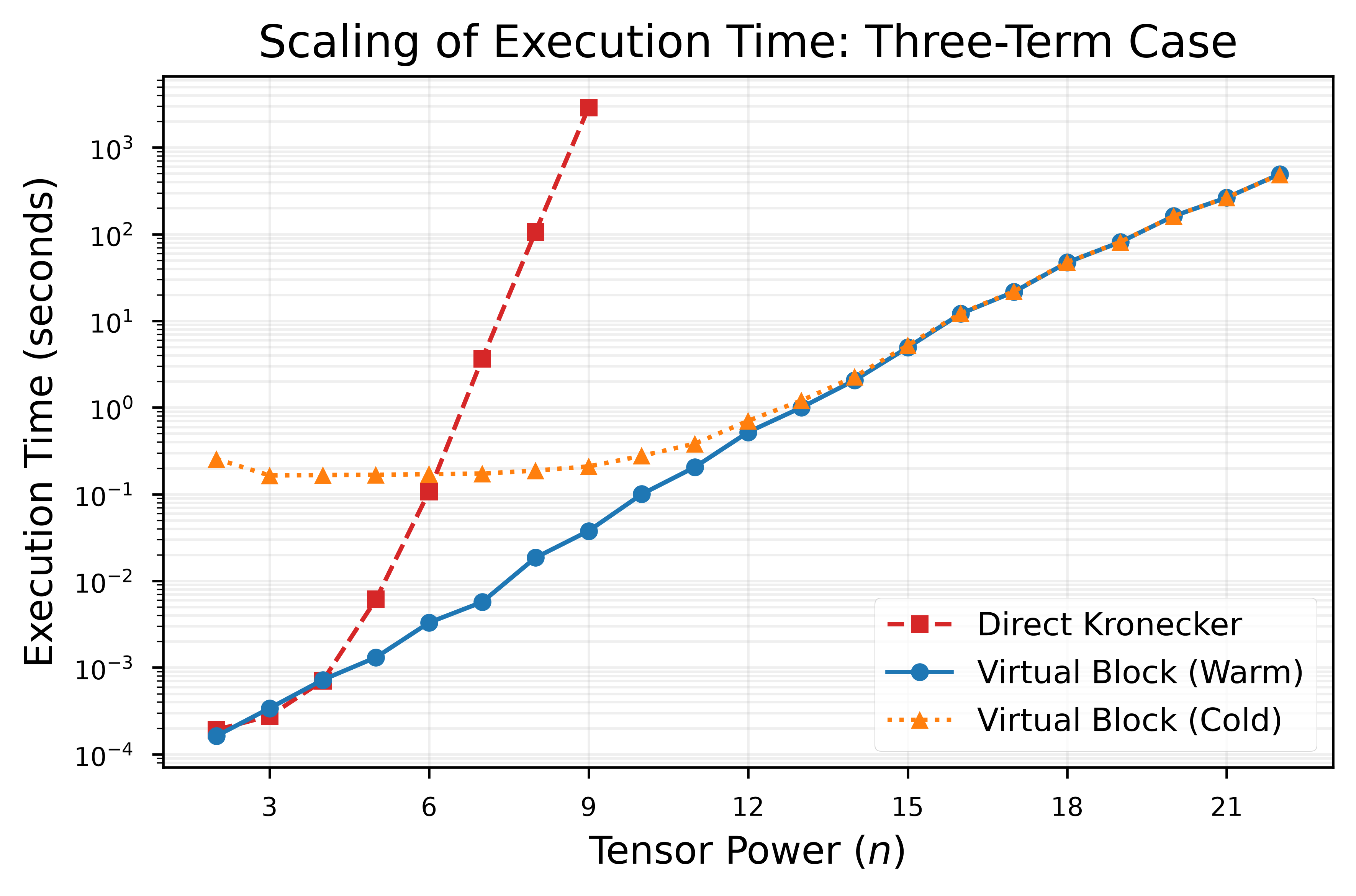}
\caption{Single-thread CPU timings for the two-term (top) and three-term (bottom) trace-norm families. The ordinate is logarithmic. Warm timings reuse the calculator and cached representation data; cold timings start a fresh Python process. The direct method explicitly forms the Kronecker matrix and computes its dense SVD. The virtual-block method scales reliably up to $n=30$. In the two-term case, after $n=26$, warm benchmarking was skipped, since package loading overhead is constant in $n$. The three-term case was only performed until $n=22$, as that was enough to show that the method is almost as fast for three-term, as for two-term, see Table \ref{tab:benchmark-selected}}.
\label{fig:benchmark-plots}
\end{figure}

The two-term and three-term virtual-block timings are close because the sizes
and multiplicities of the representation-theoretic blocks are determined by
$d$ and $n$, not by the number of summands.  At $n=22$, for example, the warm
medians are about $481$ and $494$ seconds, respectively.  The additional term
therefore adds only a small overhead in this test family, even though it
changes the algebraic character of the problem: unlike the two-term
determinant reduction discussed in the introduction, the three-term trace norm
has no comparable elementary spectral factorization.

The optional GPU backend changes only the dense SVD stage and not the representation-theoretic reduction. While the CPU benchmark establishes a rigorous, machine-independent baseline, delegating the SVD stage to the GPU in version 0.4.2 provides an approximately 9-fold ($9\times$) speedup for the largest reduced blocks. Note that GPU timings are not combined with the CPU data in Table~\ref{tab:benchmark-selected}; when utilizing the GPU, device memory can eventually become the limiting resource for the largest blocks even if host memory remains sufficient. This speedup was achieved with an NVIDIA 5070 that has 12 GB of VRAM, while other components and packages were kept the same. The VRAM was sufficient until $n=26$ with partial offloading to RAM.

\subsection{Cancellation and numerical stability}

Exactness in the representation ring does not imply numerical
well-conditioning.  For a fixed $\lambda$, put
\[
  F_{\lambda,\sigma}
  =\norm{\sum_i t_iB_{\lambda,\sigma}(A_i)}_p^p\geq0,
  \qquad
  F_\lambda=\sum_\sigma\sgn(\sigma)F_{\lambda,\sigma}.
\]
Whenever $F_\lambda\neq0$, define the a posteriori cancellation indicator
\begin{equation}\label{eq:cancellation}
  \kappa_{\mathrm{virt}}(\lambda)
  =\frac{\sum_\sigma\abs{F_{\lambda,\sigma}}}{\abs{F_\lambda}}.
\end{equation}
Large values signal loss of relative accuracy in floating-point subtraction.
For $d=3$, this is the sum of the two nonnegative terms in
\eqref{eq:d3-functional} divided by the absolute value of their difference.
An analogous global indicator is obtained from the aggregated terms in
\eqref{eq:d3-aggregated}.

The following safeguards are appropriate in both testing and production
runs:
\begin{enumerate}[leftmargin=2.2em]
  \item validate all feasible small $n$ against direct Kronecker construction;
  \item include random unitary, nonnormal, ill-conditioned, singular, and
        exactly commuting test matrices;
  \item report both absolute and relative residuals, since the true value may
        be close to zero;
  \item monitor \eqref{eq:cancellation} and repeat poorly conditioned cases in
        extended precision;
  \item use pairwise or compensated summation for the signed scalar
        contributions;
  \item use scaling or logarithmic bookkeeping for large determinant powers,
        and never silently truncate a substantially negative computed value of
        $\norm{X_n}_p^p$ to zero.
\end{enumerate}

The benchmark distinguishes cold-start loading from warm repeated evaluation.
This separation is particularly important in hypothesis-testing searches,
where many matrix families are evaluated at the same tensor powers and the
matrix-independent representation data are amortized.

\section{Discussion and outlook}
\label{sec:discussion}

The method separates three structures that are easily conflated. Schur--Weyl duality supplies an orthogonal block decomposition; the
representation ring supplies virtual identities among additive or multiplicative block functionals; and the symmetric-power formula supplies
concrete matrices. These explain both the generality and the limitations of the construction.

Starting with $\U(d)$ is essential for Schatten norms, but imposes no unitary
restriction on the inputs.  The same unitary intertwiner works for
$\GL(d,\C)$ by complexification and for singular matrices by polynomial
continuation.  Determinant characters are equally structural: removing the
last row of a partition reduces the Jacobi--Trudi determinant from size $d$
to at most $d-1$, thereby improving the block-size exponent.

The fixed-dimension polynomial reduction is exact, but the present dense
implementation is not intended to be optimal.  Structure-preserving SVDs,
randomized methods with certified error bounds, and problem-specific sparsity
may extend the practical range.  The current software already permits the SVD
stage to be carried out on a compatible GPU, while leaving the reduction and
its exact formulas unchanged.  The matrix-independent representation data can
be cached locally and distributed as a separate versioned resource.
Other direct-sum additive spectral functionals can be treated by the same
virtual identities.  Multiplicative invariants lead to products with integer
exponents. In contrast, the operator norm and other functionals governed by
a maximum rather than a sum cannot be recovered from a formal virtual
difference alone.

The benchmark release records the random seed, cold and warm timings,
software and BLAS/LAPACK information, thread settings, raw repetitions, and
direct-method residuals, and is archived independently of the manuscript
\cite{TensorpowBenchmark}.  This provides a reproducible performance baseline
without fitting an empirical complexity exponent over a short range of $n$;
the asymptotic statement remains the exact worst-case bound of
Theorem~\ref{thm:complexity}.

From the viewpoint of the motivating application, the main value of the
method is access to finite-copy expressions such as
\eqref{eq:composite-helstrom}, precisely where three or more tensor-power
terms prevent the elementary reductions available for one term or for a
two-term determinant.  The representation-theoretic compression does not
supply a single-letter asymptotic formula by itself, but it makes numerical
exploration of that difficult regime possible far beyond brute force.

\section*{Acknowledgements}

The authors thank Sloan Nietert, who developed the earlier Wolfram
Mathematica realization of the $2\times2$ method under M.W.'s supervision.
That implementation is archived at Zenodo \cite{WeinerNietertSoftware}.

\paragraph{Author contributions.}
M.W. conceived the representation-theoretic approach, including the use of
virtual representations and symmetric-power realizations. Under his supervision, M.Á.J., for his MSc thesis, developed the explicit $3\times3$ realization, designed and
implemented the algorithms and software, and carried out the numerical
experiments.  Both authors contributed to the mathematical presentation,
verified the final manuscript, and accept responsibility for its contents.
The authors are listed alphabetically.

\paragraph{Code and data availability}
The Python implementation \texttt{tensorpow} is available from its GitHub
repository and from the Python Package Index \cite{TensorpowSoftware}.  The
benchmark data, reproducibility metadata, tables, and plotting outputs used in
Section~\ref{sec:benchmark} are archived at Zenodo
\cite{TensorpowBenchmark}.  The earlier Mathematica implementation for
$2\times2$ matrices is archived separately at Zenodo
\cite{WeinerNietertSoftware}.

\section*{Declaration of generative AI and AI-assisted technologies in the manuscript preparation process}
During the preparation of this work, the authors used OpenAI ChatGPT to assist
in restructuring material from the MSc thesis into article form and to
improve the organization, language, and readability of the manuscript.  The
authors reviewed, verified, and edited all mathematical statements, proofs,
references, and numerical claims and take full responsibility for the content
of the publication.

\end{document}